\documentclass{amsart}
\usepackage{amsmath,mathrsfs}
\usepackage{amsfonts}
\usepackage{amssymb,color}
\usepackage{graphicx}
\usepackage[square,numbers]{natbib}
\usepackage{verbatim}
\usepackage{enumerate,xcolor,hyperref,comment}
\newtheorem{theorem}{Theorem}
\newtheorem{claim}[theorem]{Claim}
\newtheorem{corollary}[theorem]{Corollary}
\newtheorem{lemma}[theorem]{Lemma}
\newtheorem{proposition}[theorem]{Proposition}

\theoremstyle{definition}
\newtheorem{definition}[theorem]{Definition}

\theoremstyle{remark}
\newtheorem{remark}[theorem]{Remark}

\numberwithin{theorem}{section}
\numberwithin{equation}{section}

\def\XXint#1#2#3{{\setbox0=\hbox{$#1{#2#3}{\int}$}
     \vcenter{\hbox{$#2#3$}}\kern-.5\wd0}}

\newcommand{\dd}{\; \mathrm{d}}

\newcommand{\bbH}{\mathbb{H}}
\newcommand{\bbR}{\mathbb{R}}
\newcommand{\bbN}{\mathbb{N}}

\newcommand{\R}{\mathbb{R}}

\begin{document}
\title[$C^{m,\omega}$ Whitney extension for curves in the Heisenberg group]{A $C^{m,\omega}$ Whitney Extension Theorem for Horizontal Curves in the Heisenberg Group}

\author[Gareth Speight]{Gareth Speight}
\address[Gareth Speight]{Department of Mathematical Sciences, University of Cincinnati, 2815 Commons Way, Cincinnati, OH 45221, United States}
\email[Gareth Speight]{Gareth.Speight@uc.edu}

\author[Scott Zimmerman]{Scott Zimmerman}
\address[Scott Zimmerman]{Department of Mathematics, The Ohio State University at Marion, 1465 Mt Vernon Ave, Marion, OH 43302, United States}
\email[Scott Zimmerman]{zimmerman.416@osu.edu}


\date{\today}

\begin{abstract}
We characterize which mappings from a compact subset of $\mathbb{R}$ into the Heisenberg group can be extended to a $C^{m,\omega}$ horizontal curve for a given modulus of continuity $\omega$. We motivate our characterization by showing that the $C^{m,\omega}$ extension property fails if we instead use a more direct analogue of the conditions from the $C^{m}$ case.
\end{abstract}

\maketitle

\section{Introduction}

The Whitney extension theorem \cite{Whi34} characterizes those collections of real-valued continuous functions $F=(F^{k})_{|k|\leq m}$ defined on a compact set $K\subset \bbR^{n}$ that can be extended to a $C^{m}$ function $f$ such that the derivatives $D^{k}f$ coincide with $F^{k}$ on the set $K$. The key condition is that $F=(F^{k})_{|k|\leq m}$ must form a Whitney field, which encodes the fact that Taylor's theorem must hold if $F=(F^{k})_{|k|\leq m}$ are to extend to a $C^{m}$ mapping. There 
have been a variety of versions of Whitney's theorem for mappings with different regularity \cite{Bie80, Feff1, Feff2, Feff3, Feff4, Folland} or between different spaces \cite{FSS01, FSS03, FSS03b, JS16, SS18, VP06, Zim18, Zim21}. Whitney extension results are applied to study rectifiable sets, construct mappings with desired differentiability properties, or prove Lusin approximation results \cite{CPS21a, CPS21b, Fed69, LS16, Spe14, Whi35}. In \cite{PSZ19} the present authors together with Pinamonti proved a Whitney extension theorem for $C^{m}$ horizontal curves in the Heisenberg group. In this paper, we refine the techniques from \cite{PSZ19} to prove a Whitney extension theorem for $C^{m,\omega}$ horizontal curves, where $\omega$ is a modulus of continuity,
and we explain why the assumptions from \cite{PSZ19} are not sufficient in the $C^{m,\omega}$ case.

Carnot groups are Lie groups whose Lie algebra admits a stratification that gives rise to dilations and ensures that points can be connected by horizontal curves i.e. absolutely continuous curves with tangents in a distinguished subbundle of the tangent bundle. The Carnot-Caratheodory distance is defined by infimizing over the lengths of such curves
and equips every Carnot group with a natural left-invariant metric. 
In recent years, it has become clear that a large part of geometric analysis, geometric measure theory, and real analysis in Euclidean spaces may be generalized to the Carnot group setting. See, for example, \cite{BLU07, CDPT07, FSS01, FSS03, LPS17, Mon02, MPS17, Pan89, PS16, PS18, PSZ19}. 

The present paper focuses on the first Heisenberg group $\mathbb{H}^{1}$ (Definition \ref{Heisenberg}), which is the simplest and most often studied non-Euclidean Carnot group. It can be viewed in coordinates as $\mathbb{R}^{3}$ with a two-dimensional horizontal distribution. (While we expect our results to hold in the Heisenberg group $\mathbb{H}^{n}$ of any dimension, we will focus our attention on the first Heisenberg group $\mathbb{H}^{1}$ to keep the notation more manageable.) As mentioned above, the current authors together with Pinamonti \cite{PSZ19} characterized when a trio of collections of real valued continuous functions $(F^{k},G^{k},H^{k})_{k=0}^{m}$ can be extended to a $C^{m}$ horizontal curve in $\mathbb{H}^{1}$. There were three main conditions in the characterization. Firstly, since each of $(F^{k})_{k=0}^{m}$, $(G^{k})_{k=0}^{m}$, $(H^{k})_{k=0}^{m}$ must extend in particular to a $C^{m}$ map from $\mathbb{R}$ to $\mathbb{R}$, they must be Whitney fields of class $C^{m}$ according to Whitney's original theorem. Secondly, if the trio does indeed extend to a horizontal $C^{m}$ map, then differentiating the horizontality condition (Lemma \ref{lift}) and restricting to $K$ gives equations which must necessarily hold for each map $H^{k}$ in terms of $(F^{k})_{k=0}^{m}$ and $(G^{k})_{k=0}^{m}$. Finally, the horizontality condition shows that changes in the third component (i.e. the height) correspond to changes in the areas swept out by the first two components in the plane. This is encoded via the Taylor polynomial in the area-velocity condition described in \cite{PSZ19}.

In the present paper we characterize when a trio of collections $(F^{k},G^{k},H^{k})_{k=0}^{m}$ of continuous functions on a compact set $K\subset \mathbb{R}$ can be extended to a $C^{m,\omega}$ horizontal curve in $\mathbb{H}^{1}$, where $\omega$ is a concave modulus of continuity (Definition \ref{defmodulus}). For instance, the choice $\omega(t)=t$ corresponds to those maps which are $C^{m}$ and whose highest order derivatives are Lipschitz. Our main result is the following. For most definitions, see Section \ref{preliminaries}; the quantities $A(a,b)$ and $V_\omega (a,b)$ are defined in \eqref{Aab} and \eqref{Vab} respectively.

\begin{theorem}\label{iffomega}
Suppose $\omega$ is a modulus of continuity.
Let $K\subset \bbR$ be a compact set and $F=(F^{k})_{k=0}^{m}$, $G=(G^{k})_{k=0}^{m}$, $H=(H^{k})_{k=0}^{m}$ be collections of continuous, real-valued functions on $K$. Then there is a horizontal curve $\Gamma \in C^{m,\omega}(\mathbb{R},\mathbb{R}^3)$ such that 
$D^k \Gamma|_K = (F^k,G^k,H^k)$ for $0\leq k \leq m$ 
if and only if
\begin{enumerate}
\item $F$, $G$, and $H$ are Whitney fields of class $C^{m,\omega}$ on $K$,
\item for every $1 \leq k \leq m$ the following holds on $K$,
\[        H^k = 2 \sum_{i=0}^{k-1}  \binom{k-1}{i}  \left(F^{k-i}G^i- G^{k-i}F^i \right),\]
\item and there is a constant $C\geq 1$ such that
\[ \left| \frac{A(a,b)}{V_{\omega}(a,b)}\right| \leq C \qquad \mbox{for all }a,b\in K \mbox{ with }a<b.\]
\end{enumerate}
\end{theorem}

As discussed above, condition (1) follows from Whitney's original extension theorem, while condition (2) follows from the horizontality of the extension exactly as before. It is in condition (3) that we see a key difference when compared to the $C^{m}$ case. The area-velocity estimate requires a velocity defined in terms of $\omega$ which vanishes faster than the velocity defined in \cite{PSZ19}. As pointed out in \cite{Zim21}, applying directly the techniques of \cite{PSZ19} yields an extension only to a $C^{m,\sqrt{w}}$ horizontal curve. (See Theorem~\ref{motivation} below.)

The proof of Theorem \ref{iffomega} is similar to that of the $C^{m}$ case in \cite{PSZ19}; we use Euclidean techniques to find a $C^{m,\omega}$ extension of $F$ and $G$ then apply perturbations and a horizontal lift to define an extension of $H$. However, the perturbations and corresponding estimates must be more carefully chosen in the $C^{m,\omega}$ case.

That a new definition of velocity is indeed required is shown by the following, which is our second main result. It is proved in Section~\ref{sectionjustify} by direct construction. The quantity $V_{1}(a,b)$ mentioned here is the velocity term introduced in \cite{PSZ19} and defined in \eqref{Vhatab} below.

\begin{theorem}\label{motivation}
Suppose $\omega$ is a modulus of continuity.
There is a compact and perfect set $K \subset \mathbb{R}$, a constant $\hat{C} > 0$, and collections $F=(F^k)_{k=0}^m$, $G=(G^k)_{k=0}^m$, and $H=(H^k)_{k=0}^m$ of continuous, real valued functions on $K$
such that
\begin{enumerate}
    \item 
    $F$, $G$, and $H$ are Whitney fields of class $C^{m,\omega}$ on $K$,
    \item for every $1 \leq k \leq m$ and $x \in K$, the following holds on $K$:
    $$
        H^k = 2 \sum_{i=0}^{k-1}  \binom{k-1}{i}  \left(F^{k-i}G^i- G^{k-i}F^i \right),
    $$
    \item 
    and 
\begin{equation*}
\left| \frac{A(a,b)}{V_1(a,b)} \right| \leq \hat{C}\omega(b-a)
\quad \text{for all } a,b \in K \text{ with } a<b,
\end{equation*}
\end{enumerate}
but there is no horizontal curve $\Gamma\in C^{m,\omega^\alpha}(\R,\R^3)$ 
satisfying $\Gamma|_K = (F^0,G^0,H^0)$
for any $\alpha \in (\tfrac12,1]$.
\end{theorem}

The paper is organized as follows. In Section \ref{preliminaries} we recall the necessary background. In Section \ref{sectionnecessary} and Section \ref{sectionsufficient} we prove that the conditions for $C^{m,\omega}$ horizontal extension in Theorem \ref{iffomega} are necessary and sufficient respectively. Finally in Section \ref{sectionjustify} we prove Theorem \ref{motivation}.

\bigskip

\textbf{Acknowledgements:} This work was supported by a grant from the Simons Foundation (\#576219, G. Speight).

\section{Preliminaries}\label{preliminaries}

\subsection{The Heisenberg Group}

\begin{definition}\label{Heisenberg}
For each integer $n\geq 1$, the \emph{Heisenberg group} $\mathbb{H}^{n}$ is the Lie group represented in coordinates by $\mathbb{R}^{2n+1}$ with points denoted $(x,y,t)$ with $x,y\in \mathbb{R}^{n}$ and $t\in \mathbb{R}$. The group law is given by:
\[(x,y,t) (x',y',t')=\left(x+x',y+y',t+t'+2\sum_{i=1}^{n}(y_{i}x_{i}'-x_{i}y_{i}')\right).\]
\end{definition}

We equip $\mathbb{H}^{n}$ with left invariant vector fields
\begin{equation}\label{heisenbergvectors}
X_{i}=\partial_{x_{i}}+2y_{i}\partial_{t}, \quad Y_{i} = \partial_{y_{i}}-2x_{i}\partial_{t}, \quad 1\leq i\leq n, \quad T=\partial_{t}.
\end{equation}
Here $\partial_{x_{i}}, \partial_{y_{i}}$ and $\partial_{t}$ denote the coordinate vectors in $\mathbb{R}^{2n+1}$, which may be interpreted as operators on differentiable functions. If $[\cdot, \cdot]$ denotes the Lie bracket of vector fields, then $[X_{i}, Y_{i}]=-4T$. Thus $\mathbb{H}^{n}$ is a step-2 Carnot group with horizontal layer $\mathrm{Span}\{X_{i}, Y_{i} \colon 1\leq i\leq n\}$ and second layer $\mathrm{Span}\{T\}$.

\begin{definition}
A vector in $T_p \mathbb{R}^{2n+1}$ is \emph{horizontal} at $p \in \mathbb{R}^{2n+1}$ if it is a linear combination of the vectors $X_{i}(p), Y_{i}(p), 1\leq i\leq n$.

An absolutely continuous curve $\gamma$ in $\mathbb{R}^{2n+1}$ is \emph{horizontal} if, at almost every point $t$, the derivative $\gamma'(t)$ is horizontal at $\gamma(t)$.
\end{definition}

\begin{lemma}\label{lift}
An absolutely continuous curve $\gamma\colon [a,b]\to \mathbb{R}^{2n+1}$ is a horizontal curve in the Heisenberg group if and only if, for $t\in [a,b]$:
\[\gamma_{2n+1}(t)=\gamma_{2n+1}(a)+2\sum_{i=1}^{n}\int_{a}^{t} (\gamma_{i}'\gamma_{n+i}-\gamma_{n+i}'\gamma_{i}).\]
\end{lemma}



Lemma \ref{lift} implies that for any horizontal curve $\gamma$ we have
$$
\gamma_{2n+1}'(t) = 2\sum_{i=1}^{n} (\gamma_{i}'(t)\gamma_{n+i}(t)-\gamma_{n+i}'(t)\gamma_{i}(t)) 
\quad
\text{for a.e. } t \in [a,b].
$$ 
If we assume that $\gamma$ is $C^1$, this equality holds for every $t \in [a,b]$. If we further assume that $\gamma$ is $C^m$ for some $m \geq 1$, then, for $1 \leq k \leq m$, we may write 
\begin{equation}
\label{HigherHoriz}
D^k\gamma_{2n+1}(t) = \sum_{i=1}^{n} \mathcal{P}^k\left(\gamma_{i}(t),\gamma_{n+i}(t),\gamma_{i}'(t),\gamma_{n+i}'(t),\dots,D^k\gamma_{i}(t),D^k\gamma_{n+i}(t)\right)
\end{equation}
for all $t\in [a,b]$ where $\mathcal{P}^k$ is a polynomial determined by the Leibniz rule. If $n=1$ writing out $\mathcal{P}^k$ explicitly gives
\[\gamma_{3}^k = 2 \sum_{i=0}^{k-1}  \binom{k-1}{i}  \left(\gamma_{1}^{k-i}\gamma_{2}^i- \gamma_{2}^{k-i}\gamma_{1}^i \right).\]

Throughout this paper we work in the first Heisenberg group $\mathbb{H}^{1}$.

\subsection{The Euclidean Whitney Extension Theorem for $C^{m,\omega}$ mappings}

Throughout this paper we assume that $m\geq 1$ is an integer and $\omega$ is a modulus of continuity with the properties given in the following definition.

\begin{definition}\label{defmodulus}
A \emph{modulus of continuity} is a function $\omega\colon [0,\infty] \to [0,\infty]$ which is continuous, increasing, and concave with $\omega(0)=0$ and which is not identically 0.

Given a modulus of continuity $\omega$ and an interval $I\subset \mathbb{R}$, a map $\varphi\colon I\to \mathbb{R}$ is \emph{of class $C^{m,\omega}$} if $\varphi$ is $C^m$ and the following seminorm is finite:
\[ 
\|\varphi\|_{C^{m,\omega}(I)}:=\sup_{\substack{a,b\in I\\ a\neq b}}\frac{|D^m\varphi(b)- D^m\varphi(a)|}{\omega(|b-a|)}.\]
\end{definition}

In other words, the derivative $D^m \varphi$ is uniformly continuous on $I$ with modulus of continuity $\omega$.
The following lemma will be useful later.

\begin{lemma}\label{omegaovert}
The function $t\mapsto \omega(t)/t$ is decreasing on $(0,\infty)$.
\end{lemma}

\begin{proof}
Assume $x,y \in (0,\infty)$ with $x<y$. Since $\omega$ is concave, 
\[ \left(1-\frac{x}{y}\right)\omega(0)+\left(\frac{x}{y}\right)\omega(y)\leq \omega\left(\left(1-\frac{x}{y}\right)\cdot 0 +\left(\frac{x}{y}\right)\cdot y\right).\]
Hence $\frac{x}{y} \omega(y) \leq \omega(x)$ which gives $\frac{\omega(y)}{y}\leq \frac{\omega(x)}{x}$.
\end{proof}



\begin{definition}\label{taylor}
Suppose $K \subset \mathbb{R}$ and 
$F=(F^{k})_{k=0}^{m}$ is a collection of continuous, real-valued functions defined on $K$.
Given $a\in K$, 
the \emph{Taylor polynomial of order $m$ of $F$ at $a$} is defined by
$T_{a}^{m}F(x)=\sum_{k=0}^{m} \frac{F^{k}(a)}{k!}(x-a)^{k}$ for $x \in \mathbb{R}$. If $m$ or $a$ are clear from the context, we write $T_{a}F$ or even $TF$ for the Taylor polynomial.

When $f \in C^m(I)$ for an interval $I \subseteq \bbR$, the Taylor polynomial $T_a^mf$ is defined as usual using the collection $(D^k f)_{k=0}^m$.
\end{definition}


By Taylor's theorem, for all $f\in C^{m,\omega}(I)$ there is a constant $C>0$ so that
\begin{equation}\label{tayloroemgaest}
|D^{k}f(b)-T_{a}^{m-k}(D^{k}f)(b)| \leq C\omega(|b-a|) |b-a|^{m-k} \mbox { for }a,b\in I, 0\leq k\leq m.
\end{equation}
See, for example, (2) in \cite{Folland}.

\begin{definition}\label{whitneyfield}
Let $K\subset \bbR$ be compact and 
$F=(F^{k})_{k=0}^{m}$ be a collection of continuous, real-valued functions defined on $K$.

$F$ is a \emph{Whitney field of class $C^m$ on $K$} if
\[|F^{k}(b)-T_{a}^{m-k}F^{k}(b)|=o(|b-a|^{m-k})\]
for every $0\leq k\leq m$, uniformly as $|b-a|\to 0$ with $a,b \in K$.

$F$ is a \emph{Whitney field of class $C^{m,\omega}$ on $K$} if there is a constant $C>0$ so that
\[ |F^{k}(b)-T_{a}^{m-k}F^{k}(b)| \leq C\omega(|b-a|) |b-a|^{m-k}\]
for every $0\leq k\leq m$ and $a,b\in K$.
\end{definition}

The classical Whitney extension theorem can be stated as follows \cite{Whi34}.

\begin{theorem}[Classical Whitney extension theorem]\label{classicalWhitney}
Let $K$ be a closed subset of an open set $U\subset \mathbb{R}$. Then there is a continuous linear mapping $W$ from the space of Whitney fields of class $C^m$ on $K$ to $C^m(U)$ such that
\[D^{k}(WF)(x)=F^{k}(x) \quad \mbox{for $0\leq k\leq m$ and $x\in K$},\]
and $WF$ is $C^{\infty}$ on $U\setminus K$.
\end{theorem}

The following $C^{m,\omega}$ version of the classical Whitney extension theorem follows immediately from Whitney's original proof.

\begin{theorem}\label{classicalWhitneyomega}
Suppose $K\subset \mathbb{R}$ is closed and 
$F=(F^{k})_{k=0}^{m}$ is a collection of continuous, real-valued functions defined on $K$.
Then there is some $f\in C^{m,\omega}(\mathbb{R})$ satisfying $D^{k}f|_{K}=F^{k}$ for $0\leq k\leq m$ if and only if $F$ is a Whitney field of class $C^{m,\omega}$ on $K$.
\end{theorem}



\subsection{Area Discrepency and Velocity in $\mathbb{H}^{1}$}

Fix $\omega\colon [0,\infty] \to [0,\infty]$ a modulus of continuity. Let 
$F=(F^{k})_{k=0}^{m}$ and
$G=(G^{k})_{k=0}^{m}$
be collections of continuous, real-valued functions on $K$ and suppose $H:K \to \mathbb{R}$ is continuous.
Given $a,b\in K$ with $a<b$, define the \emph{area discrepancy} associated with $(F,G,H)$ as
\begin{align}\label{Aab}
 A(a,b)&:=H(b)-H(a)-2\int_{a}^{b}((T_{a}^mF)'(T_{a}^mG)-(T_{a}^mG)'(T_{a}^mF))\\
\nonumber &\qquad +2F(a)(G(b)-T_{a}^mG(b))-2G(a)(F(b)-T_{a}^mF(b)).
\end{align}
We define the associated \emph{velocity} as
\begin{equation}\label{Vab}
V_\omega(a,b):= (\omega(b-a))^{2}(b-a)^{2m} + \omega(b-a) (b-a)^{m}\int_{a}^{b}(|(T_{a}^{m}F)'|+|(T_{a}^{m}G)'|).
\end{equation}

Here we use the identifications $F(x) = F^0(x)$ and $G(x) = G^0(x)$.
Compare these definitions to those in \cite{PSZ19}.
Note $V_\omega (a,b)$ depends on the modulus of continuity $\omega$. Since $\omega$ is mostly fixed, we will typically denote $V_\omega$ simply as $V$. 

\begin{remark}
\label{LeftInvar}
Let $F=(F^{k})_{k=0}^{m}$ and
$G=(G^{k})_{k=0}^{m}$ 
be 
collections of continuous, real-valued functions on
a compact set $K\subset \mathbb{R}$
and let $H:K \to \mathbb{R}$ be continuous.
Fix a point $a\in K$. Define the collections $\hat{F} = (\hat{F}^k)_{k=0}^m$  and
$\hat{G} = (\hat{G}^k)_{k=0}^m$
on $K$
and the function 
$\hat{H}$
by
\begin{itemize}
\item $(\hat{F},\hat{G},\hat{H}) = (F(a),G(a),H(a))^{-1}(F,G,H)$,
\item $\hat{F}^k = F^k$ for $1 \leq k \leq m$,
\item $\hat{G}^k = G^k$ for $1 \leq k \leq m$.
\end{itemize}
Then $\hat{F}(a) = \hat{G}(a) = \hat{H}(a) = 0$, and 
$$
\hat{A}(a,b) = A(a,b) \quad \text{and} \quad \hat{V}(a,b) = V(a,b),
$$
where $\hat{A}(a,b)$ and $\hat{V}(a,b)$ respectively denote the area discrepancy \eqref{Aab} and the velocity \eqref{Vab} 
associated with $(\hat{F},\hat{G},\hat{H})$. The proof of this is a simple direct calculation, for instance as in \cite{PSZ19}
or Lemma~3.5 in \cite{Zim21}.
\end{remark}

\subsection{Inequalities for Polynomials}

%

The following consequences of the Markov inequality \cite{Mar89, Sha04} were proved in \cite{PSZ19}.

\begin{lemma}\label{Pbig}
Let $P$ be a polynomial of degree $m\geq 1$ and fix points $a<b$. Let $M=\max_{[a,b]}|P|$. Then there exists a closed subinterval $I\subset [a,b]$ of length at least $(b-a)/4m^2$ such that $|P(x)|\geq M/2$ for all $x\in I$.
\end{lemma}

\begin{corollary}\label{intmax}
Let $P$ be a polynomial of degree $m\geq 1$ and fix points $a<b$. Let $M=\max_{[a,b]}|P|$. Then
\[ \frac{M(b-a)}{8m^2}\leq \int_{a}^{b} |P|\leq M(b-a).\]
\end{corollary}

\section{Necessary Conditions For a $C^{m,\omega}$ Horizontal Extension} \label{sectionnecessary}

We use this section to prove Proposition \ref{nec}, beginning with the following.

\begin{lemma}\label{claimtaylorest}
Suppose $I \subset \bbR$ is a compact interval and $f \in C^{m,\omega}(I)$.
Then there is a constant $C \geq 1$ depending on $I$, $\omega$, $m$, and $f$ such that, for all $a,x \in I$,
\begin{enumerate}[(i)]
\item $|D^if(x)-D^if(a)|\leq C\omega(|x-a|)$ for $0\leq i\leq m$.
\item $|f(x)-T_a^mf(x)|\leq C\omega(|x-a|) |x-a|^m$.
\item $|D^kf(x)- D^kT_a^mf(x)|\leq  C\omega(|x-a|) |x-a|^{m-k}$ for $0\leq k\leq m$.
\end{enumerate}
\end{lemma}

\begin{proof}
We prove (i). The case $i=m$ follows from $f \in C^{m,\omega}(I)$. For $0\leq i<m$ by Taylor's theorem, for all $a,x \in I$, there is $\zeta$ between $a$ and $x$ so that
\[D^{i}f(x)=D^{i}f(a)+D^{i+1}f(a)(x-a)+\cdots + \frac{D^{m-1}f(a)}{(m-1-i)!}(x-a)^{m-1-i}+\frac{D^{m}f(\zeta)}{(m-i)!}(x-a)^{m-i}.\]
Since each $D^{s}f$ is continuous and hence bounded on $I$,
\begin{equation}\label{taylorjan19}
|D^{i}f(x)-D^{i}f(a)|\leq \sum_{s=i+1}^{m}\|D^{s}f\|_{\infty}|x-a|^{s}.
\end{equation}
Since $w(t)/t$ is decreasing by Lemma~\ref{omegaovert}, $\omega(|x-a|)/|x-a|\geq \omega( \ell(I))/\ell(I)$ so $|x-a|\leq C\omega(|x-a|)$ for a constant $C$ depending on $I$ and $\omega$. Enlarging $C$ as needed and using \eqref{taylorjan19} gives statement (i). 

Estimates (ii) and (iii) follow from \eqref{tayloroemgaest}.
%
\end{proof}

Note that in Lemma \ref{claimtaylorest} the constant $C$ can be chosen to depend only on some larger compact interval containing $I$. This is clear from the proof.

\begin{proposition}\label{nec}
Let $(f,g,h)\colon \mathbb{R}\to \mathbb{H}^{1}$ be a $C^{m,\omega}$ horizontal curve and $K\subset \mathbb{R}$ be compact.
Define $F=(D^k f|_K)_{k=0}^m$, $G=(D^k g|_K)_{k=0}^m$, $H=(D^k h|_K)_{k=0}^m$. 
Then 
\begin{enumerate}
\item $F$, $G$, $H$ are Whitney fields of class $C^{m,\omega}$ on $K$, 
\item for every $1 \leq k \leq m$, the following holds on $K$:
\begin{equation}
\label{Horiznec}
        H^k = 2 \sum_{i=0}^{k-1}  \binom{k-1}{i}  \left(F^{k-i}G^i- G^{k-i}F^i \right),
\end{equation}
\item there is a constant $C\geq 1$ so $|A(a,b) |\leq C V_{\omega}(a,b)$ for all $a,b\in K$, $a<b$.
\end{enumerate}
\end{proposition}

\begin{proof}
Suppose $f, g, h, F, G, H, K$ are as in the statement of Proposition \ref{nec}. 
Without loss of generality, we may assume that $K=[A,B]$ is a closed interval.
Indeed, if (1), (2), and (3) hold on the interval $[A,B]$, then they also hold on any compact subset.
By Theorem \ref{classicalWhitneyomega}, $F$, $G$, $H$ must be Whitney fields of class $C^{m,\omega}$ on $K$.
The lifting equation \eqref{HigherHoriz} gives
\[
 D^kh = 2 \sum_{i=0}^{k-1}  \binom{k-1}{i}  \left(D^{k-i}f D^ig- D^{k-i}g D^if \right),
\]
on $\mathbb{R}$ for $1 \leq k \leq m$. This proves Proposition \ref{nec} (1) and (2).

It remains to prove (3).
Fix $a, b \in K$, $a<b$. To simplify notation, let $Tf=T_a^mf$ and $Tg=T_a^mg$ be the Taylor polynomials of $f$ and $g$ of order $m$ at $a$. We first prove $|A(a,b)|\leq CV(a,b)$ under the assumption $f(a)=g(a)=h(a)=0$. In this case $A(a,b)$ takes the form
\[ A(a,b) = h(b)-h(a)-2 \int_{a}^{b} ((Tf)'(Tg)-(Tf)(Tg)'). \]
Since $(f,g,h)$ is a horizontal curve, we have $h(b)-h(a)=2\int_{a}^{b} (f'g-fg')$.
Hence we can estimate $|A(a,b)|$ as follows
\begin{align}\label{kareaest}
&\left| h(b)-h(a)-2\int_{a}^{b} ((Tf)'Tg-Tf(Tg)') \right|\\
&\qquad \leq 2\int_{a}^{b}|f'g-(Tf)'Tg| + 2\int_{a}^{b}|fg'-Tf(Tg)'|.\nonumber
\end{align}
Writing $\omega$ to denote $\omega(b-a)$,
we estimate the first term as follows using Lemma~\ref{claimtaylorest}:
\begin{align*}
\int_{a}^{b}|f'g-(Tf)'Tg|&\leq \int_{a}^{b}|f'-(Tf)'||g-Tg| + |f'-(Tf)'||Tg| + |g-Tg||(Tf)'|\\
&\leq C^2 \omega^{2}\cdot(b-a)^{2m} + C\omega\cdot(b-a)^{m-1}\int_{a}^{b}|Tg|\\
&\qquad + C\omega\cdot(b-a)^{m}\int_{a}^{b}|(Tf)'|
\end{align*}
for some constant $C \geq 1$ depending only on $K$, $\omega$, $f$, and $g$.
Using a similar estimate for the second term gives the following estimate of \eqref{kareaest},
\begin{align*}
& \left| h(b)-h(a)-2 \int_{a}^{b} ((Tf)'Tg-Tf(Tg)') \right|\\
&\qquad \leq 4C \omega^{2} \cdot (b-a)^{2m} + 2C\omega\cdot (b-a)^{m-1}\int_{a}^{b}(|Tf|+|Tg|)\\
&\qquad \qquad + 2C\omega \cdot (b-a)^{m}\int_{a}^{b}(|(Tf)'|+|(Tg)'|).
\end{align*}
Since $Tf(a)=f(a)=0$, we have $|Tf(x)|\leq M(b-a)$ on $[a,b]$ where we denote $M=\max_{[a,b]}|(Tf)'|$. Applying Corollary \ref{intmax} to the polynomial $(Tf)'$ gives
\[\int_{a}^{b}|Tf|\leq M(b-a)^{2}\leq 8m^{2}(b-a)\int_{a}^{b}|(Tf)'|.\]
A similar argument holds for $g$. Hence, enlarging $C$,
\begin{align*}
& \left| h(b)-h(a)-2 \int_{a}^{b} ((Tf)'Tg-Tf(Tg)') \right|\\
&\qquad \leq C\omega^{2}\cdot (b-a)^{2m} + C\omega\cdot (b-a)^{m}\int_{a}^{b}(|(Tf)'|+|(Tg)'|).
\end{align*}
This shows $|A(a,b)|\leq CV(a,b)$. 

To conclude, we now undo the assumption $f(a)=g(a)=h(a)=0$.

\begin{claim}\label{caseatorigin}
To prove $|A(a,b)|\leq CV(a,b)$ for some constant $C\geq 1$, it suffices to prove it under the assumption $f(a)=g(a)=h(a)=0$.
\end{claim}

\begin{proof}
If $f(a)$, $g(a)$, and $h(a)$ are arbitrary, define the $C^m$ horizontal curve
$$
(\hat{f},\hat{g},\hat{h}) = (f(a),g(a),h(a))^{-1}(f,g,h)
$$
on $\mathbb{R}$ and 
set 
$\hat{F}=(D^m \hat{f}|_K)_{k=0}^m$, $\hat{G}=(D^m \hat{g}|_K)_{k=0}^m$, and $\hat{H}=(D^m \hat{h}|_K)_{k=0}^m$.
Notice we have $\hat{f}(a)=\hat{g}(a)=\hat{h}(a)=0$ and the estimates listed in Lemma~\ref{claimtaylorest} remain true for the same constant with $(f,g,h)$ replaced by $(\hat{f},\hat{g},\hat{h})$. Hence $|\hat{A}(a,b)|\leq C\hat{V}(a,b)$ with $C$ depending on the constant chosen in Lemma~\ref{claimtaylorest} for $(f,g,h)$ and independent of $a,b$. By Remark~\ref{LeftInvar}, we have $A(a,b)=\hat{A}(a,b)$ and $V(a,b)=\hat{V}(a,b)$. Hence $|A(a,b)|\leq CV(a,b)$ and the claim holds.
\end{proof}

This concludes the proof of Proposition \ref{nec}.
\end{proof}

\section{Sufficiency of the Conditions for a $C^{m,\omega}$ Horizontal Extension}\label{sectionsufficient}

\begin{theorem}\label{newthm}
Let $F=(F^{k})_{k=0}^{m}$, $G=(G^{k})_{k=0}^{m}$, $H=(H^{k})_{k=0}^{m}$ be collections of continuous, real-valued functions on compact $K\subset \bbR$. Assume
\begin{enumerate}
\item $F$, $G$, $H$ are Whitney fields of class $C^{m,\omega}$ on $K$,
\item for every $1 \leq k \leq m$ the following holds on $K$:
\begin{equation}\label{Hpoly}
        H^k = 2 \sum_{i=0}^{k-1}  \binom{k-1}{i}  \left(F^{k-i}G^i- G^{k-i}F^i \right),
\end{equation}
\item there is a constant $C\geq 1$ so that $|A(a,b)|\leq C V_{\omega}(a,b)$ for all $a,b\in K$ with $a<b$.
\end{enumerate}
Then there is a horizontal curve $\Gamma \in C^{m,\omega}(\bbR,\bbR^3)$
such that $D^k\Gamma|_K = (F^k,G^k,H^k)$.
\end{theorem}

We use this section to prove Theorem \ref{newthm}. Suppose $K,F,G,H$ satisfy the assumptions in the statement of the theorem. Fix $C\geq 1$ such that
\begin{equation}\label{quotientcvg}
|A(a,b)|\leq C V(a,b) \qquad \mbox{for all }a,b\in K \mbox{ with }a<b.
\end{equation}
Let $I=[\min K, \max K]$. It suffices to find a $C^{m,\omega}$ horizontal map $(f,g,h)\colon I\to \bbH$ which extends $(F,G,H)$. Here, derivatives and continuity at the endpoints are defined using one-sided limits. Write $I\setminus K=\cup_{i=1}^{\infty}(a_i,b_i)$ for disjoint open intervals $(a_i,b_i)$ with $a_i, b_i \in K$.

Using Theorem \ref{classicalWhitneyomega}, we can choose $f,g\colon \bbR\to \bbR$ of class $C^{m,\omega}$ such that
\[D^kf(x)=F^{k}(x) \mbox{ and }D^kg(x)=G^{k}(x) \mbox{ for every }x\in K \mbox{ and }0\leq k\leq m.\]
Recall that $D^{k}(T_{a}^{m}F)(x)=\sum_{\ell=0}^{m-k}\frac{F^{k+\ell}(a)}{\ell!}(x-a)^{\ell}$ and a similar expression holds for $D^{k}(T_{a}^{m}G)(x)$. 
Using the fact $F$ and $G$ are Whitney fields of class $C^{m,\omega}$ and making $C$ larger if necessary, we may assume that for all $a, x\in K$ and $0\leq k\leq m$,
\begin{equation}\label{modulus1F}
|F^{k}(x)-D^{k}(T_{a}^{m}F)(x) | \leq C\omega (|x-a|)|x-a|^{m-k},
\end{equation}
\begin{equation}\label{modulus1G}
|G^{k}(x)-D^{k}(T_{a}^{m}G)(x) | \leq C\omega (|x-a|)|x-a|^{m-k}.
\end{equation}
Since $T_{a}^{m}F=T_{a}^{m}f$ for $a\in K$, Lemma~\ref{claimtaylorest} implies that, for $a\in K$, $x\in I$, $0\leq k\leq m$, and a possibly larger constant $C$: 
\begin{equation}\label{modulus2F}
|D^kf(x)-D^{k}(T_{a}^{m}F)(x)| \leq C\omega(|x-a|)|x-a|^{m-k},
\end{equation}
\begin{equation}\label{modulus2G}
|D^kg(x)-D^{k}(T_{a}^{m}G)(x)| \leq C\omega(|x-a|)|x-a|^{m-k}.
\end{equation}
Again using Lemma~\ref{claimtaylorest} and making $C$ larger if necessary, we can also assume that for every $0\leq k\leq m$ and $x,y\in I$:
\begin{equation}\label{modctya}
|D^kf(x)-D^kf(y)|\leq C \omega(|x-y|),
\end{equation}
\begin{equation}\label{modctyb}
|D^kg(x)-D^kg(y)|\leq C \omega(|x-y|).
\end{equation}

\begin{proposition}\label{perturb}
There exists a constant $\widetilde{C}>C$ for which the following holds. For each interval $[a_i,b_i]$, there exist $C^{\infty}$ functions $\phi, \psi\colon [a_i,b_i]\to \bbR$ such that
\begin{enumerate}
\item $D^k\phi(a_i)=D^k\phi(b_i)=D^k\psi(a_i)=D^k\psi(b_i)=0$ for $0\leq k\leq m$.
\item $\max\{|D^k\phi|, |D^k\psi|\} \leq \widetilde{C} \omega(b_i-a_i)$ for $0\leq k\leq m$ on $[a_i,b_i]$.
\item $\max\{|D^m\phi(x)-D^m\phi(y)|, |D^m\psi(x)-D^m\psi(y)|\} \leq \widetilde{C} \omega(|x-y|)$ for every pair $x,y\in [a_i,b_i]$.
\item $H(b_i)-H(a_i)=2\int_{a_i}^{b_i}(f+\phi)'(g+\psi)-(g+\psi)'(f+\phi)$.
\end{enumerate}
\end{proposition}

The proof of Proposition \ref{perturb} will require several steps. Fix an interval $[a,b]$ of the form $[a_i,b_i]$ for some $i\geq 1$.

\begin{claim}
To prove Proposition \ref{perturb}, we may assume $F(a)=G(a)=H(a)=0$.
\end{claim}

\begin{proof}
Suppose $f, g, F, G, H$ satisfy \eqref{quotientcvg}--\eqref{modctyb} where $F(a)$, $G(a)$, and $H(a)$ are arbitrary. 
Define new mappings $\hat{f}=f-f(a)$, $\hat{g}=g-g(a)$ and define collections 
$\hat{F}=(\hat{F}^k)_{k=0}^m$, $\hat{G}=(\hat{G}^k)_{k=0}^m$, and $\hat{H}=(\hat{H}^k)_{k=0}^m$
of continuous, real-valued functions on $K$ so that
$$
(\hat{F},\hat{G},\hat{H}) = (F(a),G(a),H(a))^{-1}(F,G,H),
$$
$\hat{F}^k = F^k$, $\hat{G}^k = G^k$, and $\hat{H}^k$ is chosen arbitrarily for $1 \leq k \leq m$.
Here again we use the convention $F=F^0$ etc.
It is easy to verify that the analogues of \eqref{quotientcvg}--\eqref{modctyb} hold for $\hat{f},\hat{g},\hat{F},\hat{G}$ with the same constants. If we now assume that Proposition \ref{perturb} has been proven for $\hat{f},\hat{g},\hat{F},\hat{G},\hat{H}$ satisfying 
\eqref{quotientcvg}--\eqref{modctyb} and the initial condition $\hat{F}(a)=\hat{G}(a)=\hat{H}(a)=0$, 
then we have a constant $\widetilde{C}>C$ and $C^{\infty}$ functions $\phi, \psi\colon [a,b]\to \bbR$ such that
\begin{enumerate}[(i)]
\item $D^k\phi(a)=D^k\phi(b)=D^k\psi(a)=D^k\psi(b)=0$ for $0\leq k\leq m$.
\item $\max\{|D^k\phi|, |D^k\psi|\} \leq \widetilde{C} \omega(b-a)$ for $0\leq k\leq m$ on $[a,b]$.
\item $\max\{|D^m\phi(x)-D^m\phi(y)|, |D^m\psi(x)-D^m\psi(y)|\} \leq \widetilde{C} \omega(|x-y|)$ for every pair $x,y\in [a,b]$.
\item $\hat{H}(b)-\hat{H}(a)=2\int_{a}^{b}(\hat{f}+\phi)'(\hat{g}+\psi)-(\hat{g}+\psi)'(\hat{f}+\phi)$.
\end{enumerate}
Simple calculations yield as before
\[\hat{H}(b)-\hat{H}(a)
=\hat{H}(b)
=H(b)-H(a)+2F(a)G(b)-2G(a)F(b)\]
and
\begin{align*}
2\int_{a}^{b}((\hat{f}+\phi)'(\hat{g}+\psi)-(\hat{g}+\psi)'(\hat{f}+\phi))
&= 2\int_{a}^{b}((f+\phi)'(g+\psi)-(g+\psi)'(f+\phi))\\
&\qquad  +2f(a)g(b)-2g(a)f(b).
\end{align*}
Thus (iv) is equivalent to
\[H(b)-H(a)=2\int_{a}^{b}(f+\phi)'(g+\psi)-(g+\psi)'(f+\phi).\]
This is the desired statement of (4) for the general case, so the claim is proven.
\end{proof}

Assume $F(a)=G(a)=H(a)=0$. Temporarily denote $A(a,b), V(a,b), \omega(b-a)$ by  $A, V, \omega$ respectively. 
We now have
\[A=H(b)-H(a)-2\int_{a}^{b}((TF)'(TG)-(TG)'(TF))\]
and
\[V= \omega^{2}\cdot (b-a)^{2m} + \omega \cdot (b-a)^{m} \int_{a}^{b}(|(TF)'|+|(TG)'|). \]
Define
\[ \mathcal{A}:=H(b)-H(a)-2\int_{a}^{b}(f'g-g'f).\]
A similar argument to the proof of Proposition~\ref{nec}, increasing $C$ if necessary, yields
\[ \left| \int_{a}^{b}(f'g-g'f)-\int_{a}^{b} ((Tf)'Tg-Tf(Tg)') \right| \leq CV. 
\]
Combining this with \eqref{quotientcvg} shows that for some larger constant $C$,
\begin{equation}\label{estforperturb}
|\mathcal{A}| 
= \left| H(b)-H(a)-2\int_{a}^{b} (f'g-g'f) \right| 
\leq CV.
\end{equation}
Notice Proposition \ref{perturb}(4) can be rewritten as:
\begin{equation}\label{pertforAcal}
2\int_{a}^{b}((f'\psi-\psi'f)+(\phi'g-g'\phi)+(\phi'\psi-\psi'\phi))=\mathcal{A}.
\end{equation}
Next, observe $\int_{a}^{b}(f'\psi-\psi'f)=2\int_{a}^{b}f' \psi$ for any $C^{\infty}$ function $\psi$ which vanishes at $a$ and $b$. A similar equation holds for the other terms in \eqref{pertforAcal}. 
Hence constructing $\phi$ and $\psi$ which satisfy Proposition \ref{perturb}(4) is equivalent to solving
\begin{equation}
\label{goalint}
4\int_{a}^{b}(\psi f' - \phi g' + \psi \phi')=\mathcal{A},
\end{equation}
where $\mathcal{A}$ satisfies $|\mathcal{A}|\leq CV$. We now show how to do this subject to the constraints given in Proposition~\ref{perturb}. We will divide the constructions of $\phi$ and $\psi$ into two cases.

First we construct a useful map that will be helpful in both cases.
Fix the constant $C\geq 1$ for which the estimates
\eqref{quotientcvg}--\eqref{estforperturb}
hold.

\begin{lemma}\label{mollifier}
Suppose $J \subseteq [a,b]$ is a closed interval of length at least $(b-a)/18m^2$.
There is a constant $C_0 > 0$ depending only on $m$ and $\mathrm{diam}(K)$ 
and a non-negative $C^{\infty}$ function $\eta:\mathbb{R} \to \mathbb{R}$
such that
\begin{enumerate}[(a)]
\item $\eta$ vanishes outside of $J$,
\item $\eta \geq 48m^2  \omega(b-a) (b-a)^{m}$ on the middle third of $J$,
\item $\eta' \geq 81 m^2 \omega(b-a)(b-a)^{m-1}$ on a subinterval of $J$ with length $\ell(J)/6$,
\item $\eta \leq C_0 \omega(b-a) (b-a)^m$ 
\item $|D^i\eta| \leq C_0 \omega(b-a)$ on $[a,b]$ for $0\leq i\leq m$,
\item $|D^{m}\eta(x)-D^{m}\eta(y)|\leq C_0 \omega(|x-y|)$ for all $x,y\in [a,b]$.
\end{enumerate}
\end{lemma}
\begin{proof}
Let $\varphi(x)=e^{-1/(1-x^2)}$ if $x\in (-1,1)$ and $\varphi(x)=0$ otherwise.
Fix $\hat{C}_0 \geq 1$ depending only on $m$ and $\mathrm{diam}(K)$ such that
\[ \max_{i=0,1,\ldots, m} 48e^{9/8}36^{i}m^{2i+2}( \|D^{i}\varphi\|_{\infty} +1)((\mathrm{diam}(K))^{m-i}+1) \leq \hat{C}_0.\]
Let $x_{0}$ be the midpoint of $J$. Define
\[\eta(x)=48e^{9/8}m^{2}\omega(b-a)(b-a)^{m}\varphi\left(\frac{2(x-x_{0})}{\ell(J)}\right).\]
Clearly $\eta$ vanishes outside $J$. Since $\varphi \geq e^{-9/8}$ on $[-\tfrac13,\tfrac13]$, we have in the middle third of $J$ that
\[\eta \geq  48m^{2}\omega(b-a)(b-a)^{m}.\]
It is easy to check $\varphi' \geq \tfrac{27}{32}e^{-9/8}$ on $[-\tfrac23,-\tfrac13]$ since $\varphi''\leq 0$ in this region. Hence, in the second sixth of $J$ we have
\begin{align*}
\eta'
&\geq 48 e^{9/8} m^{2}\omega(b-a)(b-a)^{m} \cdot \tfrac{27}{32}e^{-9/8} \cdot \frac{2}{\ell(J)}\\
&\geq 81 m^2 \omega(b-a)(b-a)^{m-1}.
\end{align*}

On the other hand, it is easy to verify that $\eta \leq \hat{C}_{0}\omega(b-a)(b-a)^{m}$. Since $\ell(J)\geq (b-a)/18m^{2}$, we have for all $x\in [a,b]$ and $0\leq i\leq m$,
\begin{align*}
    |D^i \eta (x)| 
    &= 
    48 e^{9/8} m^2 \omega(b-a) (b-a)^m (D^i\varphi) \left(\frac{2(x-x_0)}{\ell(J)} \right) \cdot \left( \frac{2}{\ell(J)} \right)^i\\
    &\leq
    48 e^{9/8} 36^{i} m^{2i+2} \Vert D^i \varphi \Vert_\infty \omega(b-a) (b-a)^{m-i} \\
    &\leq
    \hat{C}_0 \omega(b-a)
\end{align*}
Also, for any $x,y\in [a,b]$, using Lemma \ref{omegaovert},
\begin{align*}
    |D^m \eta (x) - D^m \eta(y)|
    &\leq
    \hat{C}_0 \omega(b-a)
    \left| D^m\varphi\left(\frac{2(x-x_0)}{\ell(J)} \right) - D^m\varphi\left(\frac{(2(y-x_0)}{\ell(J)} \right) \right|\\
    &\leq
    \hat{C}_0  \| D^{m+1} \varphi\|_{\infty} \omega(b-a) \,  \left| \frac{2(x-y)}{\ell(J)} \right| \\
     &\leq
     \hat{C}_0  36m^2 \| D^{m+1} \varphi\|_{\infty} \omega(b-a) \, \left| \frac{x-y}{b-a} \right| \\
     &\leq
     \hat{C}_0  36m^2 \| D^{m+1} \varphi\|_{\infty} \omega \left( |x-y| \right).
\end{align*}
Setting $C_0 = \hat{C}_0 36 m^2 (\| D^{m+1} \varphi\|_{\infty} + 1)$ completes the proof. 
\end{proof}

\begin{lemma}\label{fbig}
Suppose
\begin{equation}
\label{fbigassumption}
\int_{a}^{b} |(Tf)'| \geq \max \left( \int_{a}^{b} |(Tg)'|,\ CC_{0} \omega(b-a)(b-a)^{m} \right).
\end{equation}
Then there exists a $C^{\infty}$ map $\psi$ on $[a,b]$ 
satisfying
\begin{enumerate}
\item $D^i\psi(a)=D^i\psi(b)=0$ for $0\leq i\leq m$,
\item $|D^i\psi(x)|\leq CC_{0}\omega(b-a)$ on $[a,b]$ for $0\leq i\leq m$,
\item $|D^{m}\psi(x)-D^{m}\psi(y)|\leq CC_0 \omega(|x-y|)$ for $x,y\in [a,b]$,
\item $4 \int_{a}^{b} \psi f' =\mathcal{A}$.
\end{enumerate}
\end{lemma}

\begin{proof}
Since $C, C_{0}\geq 1$, we have 
\begin{equation}
\label{Abound1}
|\mathcal{A}|\leq CV\leq 3C  \omega(b-a) \cdot (b-a)^{m}\int_{a}^{b}|(Tf)'|.
\end{equation}
Applying Lemma \ref{Pbig} to $(Tf)'$ gives a closed subinterval $J\subset [a,b]$ of length at least $(b-a)/4m^{2}$ such that $|(Tf)'|\geq M/2$ in $J$, where $M=\max_{[a,b]}|(Tf)'|$. If $M=0$, then $V=0$, so $\mathcal{A}=0$ and we may choose $\psi$ identically zero.
Assume $M\neq 0$. In particular, $(Tf)'$ is not identically zero on $J$.
By Lemma~\ref{mollifier}, there exists a $C^{\infty}$ map $\eta$ on $[a,b]$ with the following properties:
\begin{enumerate}[(a)]
\item $\eta$ vanishes outside of $J$,
\item $|\eta| \geq 48m^2  \omega(b-a) (b-a)^{m}$ on the middle third of $J$,
\item $|\eta| \leq C_0 \omega(b-a) (b-a)^m$ on $[a,b]$,
\item $|D^i\eta(x)| \leq C_0 \omega(b-a)$ for $x\in [a,b]$ and $0\leq i\leq m$,
\item $|D^{m}\eta(x)-D^{m}\eta(y)|\leq C_0 \omega(|x-y|)$ for $x,y\in [a,b]$,
\item the sign of $\eta$ is chosen to be the same as the sign of $(Tf)'$ on $J$ 
\end{enumerate}
Now define $\psi$ on $[a,b]$ by scaling $\eta$ by a constant:
$$
\psi = \left( \frac{\mathcal{A}}{4 \int_a^b \eta f'} \right) \eta.
$$
That $\int_a^b \eta f'\neq 0$ will be justified below. Clearly properties (1) and (4) hold. It remains to verify (2) and (3).
We first bound $\left| \int_a^b \eta f' \right|$ from below. Since the middle third of $J$ has length at least $(b-a)/12m^{2}$ and $\eta$ has the same sign as $(Tf)'$ on $J$,
\begin{align*}
\left| \int_{a}^{b} \eta (Tf)' \right|
&=\int_J \eta (Tf)'\\
&\geq \frac{(b-a)}{12m^2} 48m^2 \omega(b-a) (b-a)^{m} \frac{M}{2}\\
&= 2 \omega(b-a)\cdot (b-a)^{m+1}  M \\
&\geq 2 \omega(b-a)(b-a)^{m} \int_{a}^{b}|(Tf)'|.
\end{align*}

Using $|f'-(Tf)'|\leq C\omega(b-a)\cdot (b-a)^{m-1}$ from \eqref{modulus2F} 
together with \eqref{fbigassumption} gives
\[\int_{a}^{b}| \eta f' -\eta(Tf)'|
\leq C C_0 (\omega(b-a) (b-a)^{m})^2
\leq \omega(b-a) (b-a)^{m} \int_{a}^{b}|(Tf)'|. \]
Combining the previous two steps gives
\begin{equation}
\label{ebound}
\left| \int_{a}^{b}\eta f' \right| 
\geq \omega(b-a) (b-a)^{m} \int_{a}^{b}|(Tf)'|.
\end{equation}
From \eqref{Abound1}, \eqref{ebound}, and Lemma~\ref{mollifier} we have
$$
|D^i\psi| 
= \frac{|\mathcal{A}|}{4 \left| \int_a^b \eta f' \right|} |D^i \eta|
\leq CC_0\omega(b-a).
$$
This proves (2). Then (3) is proven similarly, 
\begin{align*}
    |D^m\psi(x) - D^m\psi(y)| 
= \frac{|\mathcal{A}|}{4 \left| \int_a^b \eta f' \right|} |D^m \eta(x) - D^m \eta(y)|
\leq CC_0 \omega ( |x-y| ).
\end{align*}
\end{proof}

Lemma \ref{fbig} has a direct analogue involving $\phi$ with conclusion $-4 \int_{a}^{b} \phi g' =\mathcal{A}$ if 
\begin{equation}
\label{gbigassumption}
\int_{a}^{b} |(Tg)'| \geq \max \left( \int_{a}^{b} |(Tf)'|,\ C C_{0} \omega(b-a)(b-a)^{m}  \right).
\end{equation}
We now study the case in which \eqref{fbigassumption} and \eqref{gbigassumption} both fail.

\begin{lemma}\label{fgsmall}
Suppose
\[C C_{0} \omega(b-a)(b-a)^{m} > \max \left( \int_{a}^{b} |(Tf)'|,\ \int_{a}^{b} |(Tg)'| \right).\]
Then there exist $C^{\infty}$ maps $\phi$ and $\psi$ 
and a constant $C_1 \geq 1$ depending only on $C$, $m$ and $\text{diam}(K)$
such that
\begin{enumerate}
\item $D^i\psi(a)=D^i\psi(b)=D^i\phi(a)=D^i\phi(b)=0$ for $0\leq i\leq m$,
\item $\max \{|D^i\psi(x)|, |D^i\phi(x)| \} \leq 6C_1 \omega(b-a)$ on $[a,b]$ for $0\leq i\leq m$,
\item $\max \{|D^{m}\psi(x)-D^{m}\psi(y)|, |D^{m}\phi(x)-D^{m}\phi(y)|\} \leq 6C_1 \omega(|x-y|)$ 
for all $x,y \in [a,b]$,
\item $4 \int_{a}^{b} (\psi f' - \phi g' + \psi \phi') =\mathcal{A}$.
\end{enumerate}
\end{lemma}

\begin{proof}
Notice the hypotheses imply
\begin{equation}\label{Abound}
 |\mathcal{A}| \leq C  V \leq C(1+2CC_{0}) \omega(b-a)^{2}(b-a)^{2m}.
\end{equation}
Denote $B := \sqrt{6C(1+2CC_{0})}$. 
Applying Lemma~\ref{mollifier} to the interval $[a,b]$ itself and rescaling the resulting function by $B/81m^2$ gives a $C^\infty$ function $\xi$ 
and a constant $C_1 > 0$ depending only on $C$, $m$ and $\text{diam}(K)$
such that
\begin{enumerate}[(a)]
\item $D^i \xi(a) = D^i \xi(b) = 0$ for $0 \leq i \leq m$,
\item $\xi' \geq B\omega(b-a) (b-a)^{m-1}$ on a subinterval $J$ of $[a,b]$ with length $(b-a)/6$,
\item $|D^i\xi| \leq C_1 \omega(b-a)$ on $[a,b]$ for $0 \leq i \leq m$,
\item $|D^{m}\xi(x)-D^{m}\xi(y)|\leq C_1\omega(|x-y|)$ for all $x,y\in [a,b]$.
\end{enumerate}
Now we apply Lemma~\ref{mollifier} to $J$ and rescale the function by $B/48m^2$ to find a $C^\infty$ function $\eta$ such that
\begin{enumerate}[(a)]
\item $\eta$ vanishes outside of $J$,

\item $\eta\geq B\omega(b-a) (b-a)^{m}$ on the middle third of $J$ (which has length $(b-a)/18$) and is non-negative elsewhere,
\item $|D^i\eta| \leq C_1 \omega(b-a)$ on $[a,b]$ for $0 \leq i \leq m$,
\item $|D^{m}\eta(x)-D^{m}\eta(y)|\leq C_1\omega(|x-y|)$ for all $x,y\in [a,b]$.
\end{enumerate}

We will now explain how we define $\phi$ and $\psi$ in multiple cases.

\textbf{Case 1:}
$\left| \int_{a}^{b} \eta f' \right| \geq  |\mathcal{A}|/24$.

Assume $\mathcal{A}\neq 0$ otherwise there is nothing to prove. Set $\phi \equiv 0$ on $[a,b]$ and $\psi = ( \mathcal{A}/(4\int_a^b \eta f')) \eta$. 
Then (1) and (4) are clearly satisfied. We have
$$
|D^i\psi| 
= \frac{|\mathcal{A}|}{4 \left| \int_a^b \eta f' \right|} |D^i \eta|
\leq 6C_1 \omega(b-a)
\quad
\text{ on }
[a,b]
$$
which gives (2). Property (3) follows similarly.

\textbf{Case 2:}
$\left| \int_{a}^{b} \xi g' \right| \geq  |\mathcal{A}|/24$.

This is identical to the previous case with $\psi \equiv 0$ on $[a,b]$ and $\phi = -( \mathcal{A}/(4\int_a^b \xi g')) \xi$.

\textbf{Case 3:}
$\max \left\{ \left| \int_{a}^{b} \eta f' \right|, \left|\int_{a}^{b} \xi g'  \right| \right\} <  |\mathcal{A}|/24$.

We first have by \eqref{Abound}
$$
\int_{a}^{b} \eta \xi' 
= \int_J \eta \xi'
\geq \tfrac{1}{18} B^{2} \omega(b-a)^{2} (b-a)^{2m} \geq \frac{|\mathcal{A}|}{3},
$$
and hence
$ 4\int_{a}^{b} (\eta f' - \xi g' + \eta \xi') > 
|\mathcal{A}|$.
Now consider $\mathcal{F}\colon \bbR\to \bbR$ defined by 
\[\mathcal{F}(\lambda)=4\int_{a}^{b} ((\lambda \eta) f' - \xi g' + (\lambda \eta) \xi'). \]
Clearly $\mathcal{F}$ is a continuous map with $\mathcal{F}(0)= -4\int_{a}^{b}\xi g' \leq 4\left|\int_{a}^{b}\xi g'\right|  < |\mathcal{A}|/6$ and $\mathcal{F}(1)>|\mathcal{A}|$. Hence, by the intermediate value theorem, there exists $\lambda \in (0,1)$ such that $\mathcal{F}(\lambda)=|\mathcal{A}|$. If $\mathcal{A}\geq 0$ this completes the proof by setting $\phi = \xi$ and $\psi = \lambda \eta$. If $\mathcal{A}<0$, we do the same with the sign of $\eta$ reversed.
\end{proof} 

Let $\widetilde{C}=\max\{ CC_{0}, 6C_1 \}$.

\begin{proof}[Proof of Proposition \ref{perturb}]
Using Lemma \ref{fbig} and Lemma \ref{fgsmall} completes the proof of Proposition~\ref{perturb}. Indeed, if \eqref{fbigassumption} holds we may choose $\phi \equiv 0$ on $[a,b]$ and $\psi$ as in Lemma \ref{fbig}. If \eqref{gbigassumption} holds, then we use the analogue of Lemma \ref{fbig} with $f$ and $\psi$ replaced by $g$ and $\phi$ respectively. If \eqref{fbigassumption} and \eqref{gbigassumption} both fail,  use Lemma \ref{fgsmall}.
\end{proof}

We are now ready to build the $C^{m,\omega}$ horizontal extension of $(F,G,H)$, first on each subinterval $[a_{i},b_{i}]$ and then on the entire interval $I$.

\begin{lemma}
\label{buildcurve}
There is a constant $C_{3}\geq 1$ such that, for all $i\geq 1$, there is a horizontal curve $(\mathcal{F}_{i}, \mathcal{G}_{i}, \mathcal{H}_{i})\colon [a_{i},b_{i}]\to \bbH^{1}$ of class $C^{m,\omega}$ that satisfies
\begin{enumerate}
\item for $0\leq k\leq m$,
$$
\begin{array}{ccc}
D^k \mathcal{F}_{i}(a_{i})=F^{k}(a_{i}), & D^k \mathcal{G}_{i}(a_{i})=G^{k}(a_{i}), & D^k \mathcal{H}_{i}(a_{i})=H^{k}(a_{i}) \\
D^k \mathcal{F}_{i}(b_{i})=F^{k}(b_{i}), & D^k \mathcal{G}_{i}(b_{i})=G^{k}(b_{i}), & D^k \mathcal{H}_{i}(b_{i})=H^{k}(b_{i}).
\end{array}
$$
\item $|D^k \mathcal{F}_{i}(x)-F^{k}(a_{i})|\leq 2\widetilde{C}\omega(b_{i}-a_i)$ 
and 
$|D^k \mathcal{G}_{i}(x)-G^{k}(a_{i})|\leq 2\widetilde{C}\omega(b_{i}-a_i)$
for $0\leq k\leq m$ and $x\in [a_{i},b_{i}]$. 
\item $|D^m \mathcal{F}_{i}(x)-D^{m} \mathcal{F}_{i}(y)|\leq 2\widetilde{C} \omega(|x-y|)$ 
and 
$|D^m \mathcal{G}_{i}(x)-D^{m}\mathcal{G}_{i}(y)|\leq 2\widetilde{C} \omega(|x-y|)$ for all $x, y\in [a_{i},b_{i}]$.
\item $|D^m \mathcal{H}_{i}(x)-D^{m}\mathcal{H}_{i}(y)|\leq C_{3}\omega(|x-y|)$ for all $x,y\in [a_{i},b_{i}]$.
\end{enumerate}
\end{lemma}

\begin{proof}
Fix $i \in \mathbb{N}$.
Set $\mathcal{F}_{i}=f+\phi$ and $\mathcal{G}_{i}=g+\psi$ where $\phi$ and $\psi$ are chosen using Proposition \ref{perturb} for the interval $[a_{i},b_{i}]$ and $f$ and $g$ are the $C^{m,\omega}$ Whitney extensions of $F$ and $G$ respectively chosen earlier.
Define $\mathcal{H}_i$ to be the horizontal lift of $\mathcal{F}_i$ and $\mathcal{G}_i$
with starting height $H(a_i)$:
\begin{equation}
\label{Hlift}
\mathcal{H}_i(x) 
:= H(a_i) + 2\int_{a_i}^x (\mathcal{F}_i'\mathcal{G}_i - \mathcal{F}_i\mathcal{G}_i') 
\quad \text{for all } x \in [a_i,b_i].
\end{equation}
Differentiating $\mathcal{H}_i$ gives for any $1 \leq k \leq m$ the equation
\begin{equation}
\label{HorizH}
        D^k\mathcal{H}_i = 2 \sum_{j=0}^{k-1}  \binom{k-1}{j}  \left(D^{k-j}\mathcal{F}_i D^j\mathcal{G}_i- D^{k-j}\mathcal{G}_i D^j\mathcal{F}_i \right),
\end{equation}
on $[a_i,b_i]$ by the Leibniz rule.

The proofs of {\em (1)} and {\em(2)} follow those of \cite[Lemma 6.7 {\em (1), (2)}]{PSZ19} with $\beta$ and the estimate (6.5) therein replaced here by $\omega$ and \eqref{modctya} respectively.
For (3) we use \eqref{modctya} and Proposition \ref{perturb} to estimate as follows,
\begin{align*}
|D^m \mathcal{F}_{i}(x)-D^{m} \mathcal{F}_{i}(y)| &\leq |D^{m}f(x)-D^{m}f(y)|+|D^{m}\phi(x)-D^{m}\phi(y)|\\
&\leq C\omega(|x-y|)+\widetilde{C}\omega(|x-y|)\\
&\leq 2\widetilde{C}\omega(|x-y|).
\end{align*}
A similar argument applies for $|D^m \mathcal{G}_{i}(x)-D^{m} \mathcal{G}_{i}(y)|$.

The horizontality of $(\mathcal{F}_{i}, \mathcal{G}_{i}, \mathcal{H}_{i})$ follows from the definition of the horizontal lift. 
It remains to verify that $\mathcal{F}_{i}, \mathcal{G}_{i}$, and $\mathcal{H}_{i}$ are all of class $C^{m,\omega}$.
Clearly, $\mathcal{F}_{i}$ and $\mathcal{G}_{i}$ are of class $C^{m,\omega}$ since $f,g,\phi$, and $\psi$ are. To see that $\mathcal{H}_{i}$ is of class $C^{m,\omega}$, observe that it is $C^{m}$ and that $D^{m}\mathcal{H}_{i}$ is a linear combination of terms of the form $D^{l}\mathcal{F}_{i}D^{m-l}\mathcal{G}_{i}$ for $0\leq l\leq m$, where both the number of terms and the coefficients are bounded. We then observe that for any $c<d$ in $[a_{i},b_{i}]$,
\begin{align*}
& |D^{l}\mathcal{F}_{i}(d)D^{m-l}\mathcal{G}_{i}(d)-D^{l}\mathcal{F}_{i}(c)D^{m-l}\mathcal{G}_{i}(c)| \\
&\qquad \leq |D^{m-l}\mathcal{G}_{i}(d)||D^{l}\mathcal{F}_{i}(d)-D^{l}\mathcal{F}_{i}(c)| + |D^{l}\mathcal{F}_{i}(c)||D^{m-l}\mathcal{G}_{i}(d)-D^{m-l}\mathcal{G}_{i}(c)|.
\end{align*}
Note $D^{l}\mathcal{F}_{i},D^{m-l}\mathcal{G}_{i}$ are bounded independently of $i$, which follows from {\em (2)} and the fact that $F^{k}$ and $G^{k}$ are continuous on the compact set $K$. If $l<m$ we estimate
\[|D^{l}\mathcal{F}_{i}(d)-D^{l}\mathcal{F}_{i}(c)| \leq \int_{c}^{d} |D^{l+1}\mathcal{F}_{i}(t)|\dd t \]
which is bounded by a multiple of $d-c$, and hence by a multiple of $\omega(|d-c|)$ using Lemma \ref{omegaovert}. If $l=m$ then the estimate of $|D^{l}\mathcal{F}_{i}(d)-D^{l}\mathcal{F}_{i}(c)|$ follows directly from Proposition \ref{perturb}. A similar argument gives the estimate of $|D^{m-l}\mathcal{G}_{i}(d)-D^{m-l}\mathcal{G}_{i}(c)|$. This shows $(\mathcal{F}_{i}, \mathcal{G}_{i}, \mathcal{H}_{i})$ is a $C^{m,\omega}$ horizontal curve. This also proves (4), since the relevant estimates were independent of $i$.
\end{proof}

Recall $I=[\min K,\ \max K]$ and $I\setminus K = \cup_{i\geq 1} (a_{i},b_{i})$, where the intervals $(a_{i},b_{i})$ are disjoint and $a_{i},b_{i}\in K$.

\begin{proposition}\label{conclusion}
Define the curve $\Gamma = (\mathcal{F},\mathcal{G},\mathcal{H}) \colon I\to \mathbb{H}^{1}$ as follows:
\[ \Gamma(x) := (F(x),G(x),H(x)) \quad \mbox{if }x \in K\]
and
\[\Gamma(x) := (\mathcal{F}_i(x),\mathcal{G}_i(x),\mathcal{H}_i(x)) \quad \mbox{if }x \in (a_i,b_i) \mbox{ for some }i \geq 1.\]
Then $\Gamma$ is a $C^{m,\omega}$ horizontal curve in $\bbH^{1}$ with
\[D^k \mathcal{F}(x)=F^{k}(x),\qquad D^k \mathcal{G}(x)=G^{k}(x), \qquad D^k \mathcal{H}(x)=H^{k}(x)\]
 for all $x\in K$ and $0\leq k\leq m$.
\end{proposition}

\begin{proof}
Clearly the curve $\Gamma$ is $C^{m,\omega}$ in the subintervals $(a_{i}, b_{i})$. Define maps $\gamma^k$ on $K$ for $0 \leq k \leq m$ by $\gamma^k=(F^k,G^k,H^k)$. 
With this notation, we have to show that $\Gamma$ is a $C^{m,\omega}$ horizontal curve and $D^k \Gamma|_{K}=\gamma^{k}$ for $0\leq k\leq m$.

The proofs that $\Gamma$ is a $C^{m}$ horizontal curve and that $D^k \Gamma|_{K}=\gamma^{k}$ for $0\leq k\leq m$
are identical to the proof of these facts in \cite[Proposition 6.8]{PSZ19}.
It remains to verify that there is a constant $\hat{C}>0$ such that
\begin{equation}
    \label{e-unifcont}
|D^{m}\Gamma(x)-D^{m}\Gamma(y)|\leq \hat{C}\omega(|x-y|)
\quad
\mbox{for all }x,y \in I.
\end{equation}
We do so by dividing into several cases.

\emph{Suppose $x,y \in K$.} Then \eqref{e-unifcont} follows from the fact that $F,G,H$ are Whitney fields of class $C^{m,\omega}$ with $\hat{C} = C$.

\emph{Suppose $x\in K$ and $y\notin K$.} Then $y\in (a_{i},b_{i})$ for some $i\geq 1$. Assume $x\leq a_{i}$, as the argument is similar if $x\geq b_{i}$. Using Lemma \ref{buildcurve}
and the definition of a Whitney field of class $C^{m,\omega}$,
we may choose a constant $\hat{C} \geq 0$ depending only on $\tilde{C}$ and $C$ such that
\begin{align*}
|D^{m}\Gamma(x)-D^{m}\Gamma(y)| \leq |\gamma^{m}(x)-\gamma^{m}(a_{i})|+|D^{m}\Gamma(a_{i})-D^{m}\Gamma(y)| \leq \hat{C} \omega(|x-y|).
\end{align*}

\emph{Suppose $x, y\notin K$ and $x,y\in (a_{i},b_{i})$ for some $i\geq 1$.} Then \eqref{e-unifcont} 
follows directly from Lemma \ref{buildcurve} with $\hat{C} = \max \{2\sqrt{3}\tilde{C}, \sqrt{3} C_3\}$.

\emph{Suppose $x, y\notin K$, $x\in (a_{i},b_{i})$ and $y\in (a_{j},b_{j})$ with $i\neq j$.} Assume $b_{i}<a_{j}$ as the opposite case is similar. Then we estimate as above
\begin{align*}
|D^{m}\Gamma(x)-D^{m}\Gamma(y)| &\leq |D^{m}\Gamma(x)-D^{m}\Gamma(b_{i})| + |\gamma^{m}(b_{i})-\gamma^{m}(a_{j})| \\
& \qquad + |D^{m}\Gamma(a_{j})-D^{m}\Gamma(y)|\\
&\leq \hat{C} \omega(|x-y|)
\end{align*}
for some $\hat{C}>0$ depending only on $\tilde{C}$ and $C$.
\end{proof}

\begin{proof}[Proof of Theorem \ref{newthm}]
Theorem \ref{newthm} follows directly from Proposition \ref{conclusion}.
\end{proof}

\begin{proof}[Proof of Theorem \ref{iffomega}]
Proposition \ref{nec} and Theorem \ref{newthm} yield Theorem \ref{iffomega}.
\end{proof}

\section{Construction Justifying Conditions for a $C^{m,\omega}$ Whitney Theorem}\label{sectionjustify}

In this section we motivate the conditions given in Theorem \ref{iffomega}. In particular, we show that Theorem \ref{iffomega} does not hold if one uses the definition of velocity given in \cite{PSZ19, Zim21}. Recall from these references that, for collections of continuous real-valued functions $(F,G,H)$ on a compact set $K$ and $a,b\in K$ with $a<b$, the velocity was defined as
\begin{equation}\label{Vhatab}
V_1 (a,b):= (b-a)^{2m} + (b-a)^{m} \int_{a}^{b} \left(|(T_{a}^mF)'|+|(T_{a}^mG)'| \right).
\end{equation}

The following is a rewording of Theorem~1.1 from \cite{PSZ19}.
Note in particular the drop in regularity for the resulting extended curve. Also, if $\omega$ is a modulus of continuity satisfying our conditions then so is $\omega^{\alpha}$ for any $0<\alpha\leq 1$.

\begin{theorem}[\cite{PSZ19}]
\label{t-HeisWhitLip}
Suppose $K \subseteq \R$ is compact and 
$F=(F^k)_{k=0}^m$, $G=(G^k)_{k=0}^m$, and $H=(H^k)_{k=0}^m$ are collections of continuous, real valued functions on $K$.
If there is a constant $\hat{C}>0$ such that
\begin{enumerate}
    \item 
    $F$, $G$, and $H$ are Whitney fields of class $C^{m,\omega}$ on $K$,
    \item for every $1 \leq k \leq m$, the following holds on $K$: $$
        H^k = 2 \sum_{i=0}^{k-1}  \binom{k-1}{i}  \left(F^{k-i}G^i- G^{k-i}F^i \right),
    $$
    \item 
    and 
\begin{equation*}
\left| \frac{A(a,b)}{V_1(a,b)} \right| \leq \hat{C}\omega(b-a)
\quad \text{for all } a,b \in K \text{ with } a<b,
\end{equation*}
\end{enumerate}
then there is a horizontal curve $\Gamma \in C^{m,\sqrt{\omega}}(\R,\R^3)$  
satisfying $D^k \Gamma|_K = (F^k,G^k,H^k)$
for $0 \leq k \leq m$
\end{theorem}

While this is not the exact statement of Theorem~1.1 in \cite{PSZ19}, it follows immediately from the same proof. For a brief explanation, see Theorem~2.6 in \cite{Zim21}.

Compare condition (3) here with condition (3) from Theorem~\ref{iffomega}.
According to the following example, condition (3) in Theorem~\ref{t-HeisWhitLip} is strictly weaker than condition (3) in Theorem~\ref{iffomega}.
More precisely, we will construct a compact set $K \subset \mathbb{R}$ and collections $F=(F^k)_{k=0}^m$, $G=(G^k)_{k=0}^m$, and $H=(H^k)_{k=0}^m$ of continuous, real valued functions on $K$ which satisfy conditions (1) and (2) of Theorem~\ref{iffomega} and condition (3) of Theorem~\ref{t-HeisWhitLip}
but for which no horizontal $C^{m,\omega}$ extension exists.


\begin{theorem}[Restatement of Theorem \ref{motivation}]
Suppose $\omega$ is a modulus of continuity.
There is a compact and perfect set $K \subset \mathbb{R}$, a constant $\hat{C} > 0$, and collections $F=(F^k)_{k=0}^m$, $G=(G^k)_{k=0}^m$, and $H=(H^k)_{k=0}^m$ of continuous, real valued functions on $K$
so that
\begin{enumerate}
    \item 
    $F$, $G$, and $H$ are Whitney fields of class $C^{m,\omega}$ on $K$,
    \item for every $1 \leq k \leq m$ and $x \in K$, the following holds on $K$:
    $$
        H^k = 2 \sum_{i=0}^{k-1}  \binom{k-1}{i}  \left(F^{k-i}G^i- G^{k-i}F^i \right),
    $$
    \item 
    and 
\begin{equation*}
\left| \frac{A(a,b)}{V_1(a,b)} \right| \leq \hat{C}\omega(b-a)
\quad \text{for all } a,b \in K \text{ with } a<b,
\end{equation*}
\end{enumerate}
but there is no horizontal curve $\Gamma\in C^{m,\omega^\alpha}(\R,\R^3)$ 
satisfying $\Gamma|_K = (F,G,H)$
for any $\alpha \in (\tfrac12,1]$.
\end{theorem}

\begin{proof}[Proof of Theorem \ref{motivation}]
Define 
$$
K = \bigcup_{n = 0}^\infty [c_n,d_n] \cup \{1\}
\quad
\text{where }
[c_n,d_n] := \left[ 1-2^{-n},1 - \tfrac34 2^{-n} \right].
$$
Define the real valued function $H^0$ on $K$ as follows:
$$
H^0(t) = \begin{cases}
4^{-mn}\omega(2^{-(n+2)}) & \text{if } t \in [c_n,d_n] \\
0 & \text{if } t = 1.
\end{cases}
$$
Set $F^k = G^k = 0$ on $K$ for $0 \leq k \leq m$, and set $H^k = 0$ on $K$ for $1 \leq k \leq m$.
We will first verify that the collections $F=(F^k)_{k=0}^m$, $G=(G^k)_{k=0}^m$, and $H=(H^k)_{k=0}^m$ are Whitney fields of class $C^{m,\omega}$.
(The collections $F=(F^k)_{k=0}^m$ and $G=(G^k)_{k=0}^m$ are clearly Whitney fields of class $C^{m,\omega}$ on $K$ since they are constantly 0 there.)

Since $T^m_aH = H^0$ for any $a \in K$, we need only check that $|H^0 (b) - H^0(a)|/ |b-a|^m$ is uniformly bounded by $C\omega(|b-a|)$ for some constant $C>0$ and all $a,b \in K$.
This is similar to the proof of Proposition~4.1 in \cite{PSZ19}.
Fix $a,b \in K$. 
If $a$ and $b$ lie in the same interval $[c_k,d_k]$, then $|H^0 (b) - H^0(a)| = 0$.
If $a$ and $b$ lie in different intervals $[c_k,d_k]$ and $[c_\ell,d_\ell]$ (say $\ell > k$), then, since $\omega$ is non-decreasing and since $|b-a| \geq (c_\ell - d_k) \geq (c_{k+1} - d_k) = 2^{-(k+2)}$,
we get
\begin{align*}
    \frac{\left| H^0 (b) - H^0(a)\right|}{|b-a|^m}
    &\leq
    \frac{4^{-mk}\omega(2^{-(k+2)}) - 4^{-m\ell}\omega(2^{-(\ell+2)})}{(c_\ell - d_k)^m}\\
    &\leq
    \frac{4^{-mk}}{(2^{-(k+2)})^m}  \omega(2^{-(k+2)})\\
    &=
    4^m \left(\tfrac12\right)^{mk} \omega(2^{-(k+2)})
    \leq 
    4^m \omega(|b-a|).
\end{align*}
A similar argument holds when either $a$ or $b$ is 1, so $H$ is a Whitney field of order $C^{m,\omega}$ on $K$.
We will now verify the following facts.
\begin{enumerate}
\item For any $a,b \in K$,
$$
\left| \frac{A(a,b)}{V_1(a,b)} \right| \leq 16^m \omega(|b-a|).
$$
\item There is no horizontal curve $\Gamma \in C^{m,\omega^\alpha}(\mathbb{R},\mathbb{R}^3)$ such that $\Gamma = (F,G,H)$ on $K$ for any $\alpha \in \left(\tfrac12 , 1 \right]$.
\end{enumerate}

\begin{proof}[Proof of (1)]
Fix $a,b \in K \setminus \{1\}$ with $a < b$.
By definition, $A(a,b) = H^0(b) - H^0(a)$ and $V_1(a,b) = (b-a)^{2m}$.
Now, there are some
$k,\ell \in \mathbb{N}$
so that $a \in [c_k,d_{k}]$ and $b \in [c_\ell,d_{\ell}]$
and $k \leq \ell$.
If $k = \ell$, there is nothing to prove 
since $H^0$ is constant on the interval $[c_k,d_k]$.
Suppose $k < \ell$.
As above, we have 
\begin{align*}
\left| \frac{A(a,b)}{V_1(a,b)} \right|
=
\frac{|H^0(b) - H^0(a)|}{(b-a)^{2m}}
\leq 
\frac{4^{-mk}}{(2^{-(k+2)})^{2m}} \omega(2^{-(k+2)})
\leq
16^m \omega(b-a)
\end{align*}
since $\omega$ is non-decreasing.
A similar proof works when $b = 1$.
\end{proof}

\begin{proof}[Proof of (2)]
Suppose $\alpha \in \left(\tfrac12, 1\right]$ and $\Gamma \in C^{m,\omega^\alpha}(\mathbb{R},\mathbb{R}^3)$ is a horizontal curve satisfying $\Gamma|_K = (F,G,H)$.
We will show that there is no constant $C > 0$ such that $|A(a,b)| \leq C V_{\omega^\alpha}(a,b)$
for all $a,b \in K$ with $a<b$, and this will contradict condition {\em (3)} in Proposition~\ref{nec}. Here, we will use the fact that $\omega^{\alpha}$ is also a modulus of continuity satisfying our hypotheses.

Write $\Gamma = (f,g,h)$. 
Since $\Gamma$ must be constant and lie on the $z$-axis on each interval $[c_n,d_n]$ and since $\Gamma$ is $C^m$, it follows that 
$D^k f = D^k g = 0$ on $K$ for $0 \leq k \leq m$.
In particular, this means that
$T_a^mf$ and $T_a^mg$ are constantly equal to 0 for any $a \in K$.
Therefore, Taylor's theorem gives a constant $C>0$ such that
\begin{enumerate}
    \item $|f(x)| \leq C\omega(|x-a|)^\alpha (b-a)^m$ and $|g(x)| \leq C\omega(|x-a|)^\alpha (b-a)^m$,
    \item $|f'(x)| \leq C\omega(|x-a|)^\alpha (b-a)^{m-1}$ and $|g'(x)| \leq C\omega(|x-a|)^\alpha (b-a)^{m-1}$
\end{enumerate}
for any $a,b \in K$ and $x$ between $a$ and $b$.
Therefore, 
recalling 
$2^{-(n+2)} = (c_{n+1} - d_n)$
so that 
$(c_{n+1} - d_n)^{2m} = 4^{-m(n + 2)}$ for any $n \in \mathbb{N}$, we must have
\begin{align*}
\left| \frac{A(d_n,c_{n+1})}{V_{\omega^\alpha}(d_n,c_{n+1})} \right|
&=
\frac{h(d_n) - h(c_{n+1})}{\omega(c_{n+1}-d_n)^{2\alpha}(c_{n+1}-d_n)^{2m}}\\
&=
\frac{4^{-mn}\omega(2^{-(n+2)}) - 4^{-m(n+1)}\omega(2^{-(n+3)})}{\omega(2^{-(n+2)})^{2\alpha} 4^{-m(n+2)}}\\
&\geq 
\frac{4^{-mn} - 4^{-m(n+1)}}{4^{-m(n + 2)}} \cdot \frac{\omega(2^{-(n+2)})}{\omega(2^{-(n+2)})^{2\alpha}}\\
&=
16^m (1-4^{-m}) \omega(2^{-(n+2)})^{1-2\alpha}.
\end{align*}
Since $\alpha > \tfrac12$, $1-2\alpha < 0$, and so $\omega(2^{-(n+2)})^{1-2\alpha} \to \infty$ as $n \to \infty$.
In other words, $\Gamma$ fails condition {\em (3)} of Proposition~\ref{nec}, and this leads to a contradiction.
Therefore, no such $\Gamma$ can exist.
\end{proof}
The proof of the theorem is complete.
\end{proof}

\end{document}